\pdfoutput=1
\RequirePackage[l2tabu, orthodox]{nag}
\documentclass[11pt,headsepline=true,toc=flat]{scrartcl}
\usepackage[colorlinks=true]{hyperref}

\usepackage[utf8]{inputenc}

\usepackage{xspace, xifthen, enumitem}

\usepackage{amssymb, amsmath, amsthm, thmtools, nth}

\usepackage{graphicx}

\usepackage{etoolbox}

\usepackage{tikz}
\usetikzlibrary{arrows,decorations.markings,chains,calc,matrix,cd}
\tikzset{>=cm to}

\usetikzlibrary{external}

\usepackage{environ}
\makeatletter
\def\tikzcd@[#1]{%
  \begin{tikzpicture}[/tikz/commutative diagrams/.cd,every diagram,#1]%
  \ifx\arrow\tikzcd@arrow%
    \pgfutil@packageerror{tikz-cd}{Diagrams cannot be nested}{}%
  \fi%
  \let\arrow\tikzcd@arrow%
  \let\ar\tikzcd@arrow%
  \def\rar{\tikzcd@xar{r}}%
  \def\lar{\tikzcd@xar{l}}%
  \def\dar{\tikzcd@xar{d}}%
  \def\uar{\tikzcd@xar{u}}%
  \def\urar{\tikzcd@xar{ur}}%
  \def\ular{\tikzcd@xar{ul}}%
  \def\drar{\tikzcd@xar{dr}}%
  \def\dlar{\tikzcd@xar{dl}}%
  \global\let\tikzcd@savedpaths\pgfutil@empty%
  \matrix[%
    /tikz/matrix of \iftikzcd@mathmode math \fi nodes,%
    /tikz/every cell/.append code={\tikzcdset{every cell}},%
    /tikz/commutative diagrams/.cd,every matrix]%
  \bgroup}

\def\endtikzcd{%
  \pgfmatrixendrow\egroup%
  \pgfextra{\global\let\tikzcdmatrixname\tikzlastnode};%
  \tikzcdset{\the\pgfmatrixcurrentrow-row diagram/.try}%
  \begingroup%
    \pgfkeys{
      /handlers/first char syntax/the character "/.initial=\tikzcd@forward@quotes,%
      /tikz/edge quotes mean={%
        edge node={node [execute at begin node=\iftikzcd@mathmode$\fi,
                         execute at end node=\iftikzcd@mathmode$\fi,
                         /tikz/commutative diagrams/.cd,every label,##2]{##1}}}}%
    \let\tikzcd@errmessage\errmessage
    \def\errmessage##1{\tikzcd@errmessage{##1^^J...^^Jl.\tikzcd@lineno\space%
        I think the culprit is a tikzcd arrow in cell \tikzcd@currentrow-\tikzcd@currentcolumn}}%
    \tikzcd@before@paths@hook%
    \tikzcd@savedpaths%
  \endgroup%
  \end{tikzpicture}%
  \ifnum0=`{}\fi}

\NewEnviron{mytikzcd}[1][]{%
  \def\@temp{\tikzcd@[#1]\BODY}%
  \expandafter\@temp\endtikzcd
}
\makeatother

\def\temp{&} \catcode`&=\active \let&=\temp

\usepackage[lining]{libertine}
\usepackage[T1]{fontenc}
\usepackage{textcomp}
\usepackage[varqu,varl]{inconsolata}
\usepackage[libertine]{newtxmath}
\usepackage{bm}
\usepackage{mathtools}
\mathtoolsset{mathic}

\usepackage[cal=boondoxo]{mathalfa}
\usepackage{mathrsfs}

\usepackage[all]{hypcap}

\usepackage{csquotes}
\usepackage[english]{babel}
\usepackage[nodayofweek]{datetime}

\usepackage[protrusion=true]{microtype}

\usepackage[bibencoding=utf8,style=alphabetic,citestyle=alphabetic,backref=true,hyperref=true,giveninits=true,doi=true]{biblatex}
\renewcommand*{\bibfont}{\normalfont\footnotesize}
\renewbibmacro{in:}{%
  \ifentrytype{article}{}{\printtext{\bibstring{in}\intitlepunct}}}
\renewrobustcmd*{\bibinitdelim}{\,}
\AtEveryBibitem{%
  \clearfield{pagetotal}%
}

\usepackage{pdfpages}

\declaretheoremstyle[spaceabove=\topsep,spacebelow=\topsep,headfont=\normalfont\scshape,notefont=\normalfont\mdseries,notebraces={(}{)},bodyfont=\normalfont,postheadspace=5pt plus 1pt minus 1pt]{scdef}
\declaretheoremstyle[spaceabove=\topsep,spacebelow=\topsep,headfont=\normalfont\scshape,notefont=\normalfont\mdseries,notebraces={(}{)},bodyfont=\itshape,postheadspace=5pt plus 1pt minus 1pt]{scthm}

\declaretheorem[style=scdef,numberwithin=section,   name=Definition,refname={Definition,Definitions},Refname={Definition,Definitions}]{definition}
\declaretheorem[style=scdef,sharenumber=definition, name=Remark,refname={Remark,Remarks},Refname={Remark,Remarks}]{remark}

\declaretheorem[style=scthm,sharenumber=definition, name=Theorem,refname={Theorem,Theorems},Refname={Theorem,Theorems}]{theorem}
\declaretheorem[style=scthm,sharenumber=definition, name=Lemma,refname={Lemma,Lemmas},Refname={Lemma,Lemmas}]{lemma}

\declaretheorem[style=scthm,sharenumber=definition, name=Proposition,refname={Proposition,Propositions},Refname={Proposition,Propositions}]{proposition}

\undef\Re
\undef\Im

\newcommand*{\isom}{\cong}

\newcommand*{\wequiv}{\simeq}

\makeatletter
\let\@oldsubset=\subset
\def\@subsethelper#1#2{\mathrel{\raisebox{.5pt}{$#1\@oldsubset$}}\xspace}
\DeclareRobustCommand*{\subset}{\mathpalette\@subsethelper\relax}

\let\@oldotimes=\otimes
\def\@otimeshelper#1#2{\mathrel{\raisebox{.5pt}{$#1\@oldotimes$}}\xspace}
\DeclareRobustCommand*{\otimes}{\mathpalette\@otimeshelper\relax}
\makeatother

\tikzset{/tikz/commutative diagrams/arrows={thin}}

\renewcommand*{\cal}[1]{\ensuremath{\mathcal{#1}}\xspace}
\newcommand*{\f}[1]{\ensuremath{\mathfrak{#1}}\xspace}

\newcommand*{\ie}{i.\,e.}

\undef\lrcorner
\newcommand{\lrcorner}{\mathord{\vrule height 0.1ex depth 0pt width 1ex\vrule height 1.3ex depth 0pt width
0.1ex}}

\def\<#1>{\left\langle #1 \right\rangle}

\undef\AA
\undef\SS
\renewcommand*{\do}[1]{\expandafter\def\csname#1#1\endcsname{\ensuremath{\mathbb{#1}}\xspace}}
\docsvlist{A,B,C,D,E,F,G,H,I,J,K,L,M,N,O,P,Q,R,S,T,U,V,W,X,Y,Z}

\let\to\longrightarrow

\setlist[enumerate]{label={\normalfont \rmfamily(\roman*)}, nosep}
\setlist[itemize]{nosep}

\usepackage{todonotes}

\undef\lim
\DeclareMathOperator*{\lim}{\textnormal{lim}}
\DeclareMathOperator*{\colim}{\textnormal{colim}}
\DeclareMathOperator*{\hocolim}{\textnormal{hocolim}}
\DeclareMathOperator*{\holim}{\textnormal{holim}}

\undef\hom
\DeclareMathOperator{\hom}{\textnormal{Hom}}
\DeclareMathOperator{\map}{\textnormal{Map}}

\DeclareMathOperator{\Spec}{\textnormal{Spec}}

\newcommand{\BGL}{\mathord{\textnormal{BGL}}}

\newcommand{\GL}{\mathord{\textnormal{GL}}}
\newcommand{\Gr}[2][]{\ifthenelse{\isempty{#1}}{%
  \mathord{\textnormal{Gr}}_{#2}}{%
  \operatorname{\textnormal{Gr}}_{#2}(#1)}}

\newcommand{\smashp}{\wedge}

\newcommand{\Tsusp}[1][]{\Sigma_{\PP^1}\ifthenelse{\isempty{#1}}{^\infty}{^{#1}}}
\newcommand{\transfer}[1][]{\operatorname{tr}\ifthenelse{\isempty{#1}}{}{_{#1}}}
\newcommand{\Tr}[1]{\operatorname{Tr}(#1)}
\newcommand{\proj}[1][]{\operatorname{\textnormal{proj}}\ifthenelse{\isempty{#1}}{}{_{#1}}}
\newcommand{\one}[1]{\mathbf{1}_{#1}}

\DeclareMathOperator{\SH}{\cal{S\mkern-1mu H}}
\undef\H
\DeclareMathOperator{\H}{\cal{H}}
\DeclareMathOperator{\Ind}{\textnormal{Ind}}

\newcommand{\Sm}[1]{\operatorname{\textnormal{Sm}}_{#1}}
\newcommand{\sm}{\textnormal{sm}}

\newcommand{\PrL}{\cal{P\mkern-4mu r}^{\textnormal{L}}}
\newcommand{\PrR}{\cal{P\mkern-4mu r}^{\textnormal{R}}}

\newcommand{\Op}{\textnormal{op}}

\DeclareMathOperator{\id}{\textnormal{id}}

\DeclareMathOperator{\unit}{\textnormal{unit}}
\DeclareMathOperator{\counit}{\textnormal{counit}}
\DeclareMathOperator{\Ex}{\textnormal{Ex}}

\DeclareMathOperator{\incl}{\textnormal{incl}}

\newcommand{\Deloop}[2][]{\ifthenelse{\isempty{#1}}{\textnormal{\bfseries B}#2}{\textnormal{\bfseries B}^{#1}#2}}
\newcommand{\deloop}[2][]{\ifthenelse{\isempty{#1}}{\textnormal{B}#2}{\textnormal{B}^{#1}#2}}

\DeclareMathOperator{\ev}{\textnormal{ev}}

\newcommand{\kanloop}[1]{\ifthenelse{\isempty{#1}}{\operatorname{\GG}}{\GG#1}}

\newcommand*{\infcat}{\(\infty\)–category\xspace}
\newcommand*{\infcats}{\(\infty\)–categories\xspace}

\newcommand{\ho}[1]{\textnormal{h}#1}

\DeclareMathOperator{\coev}{\textnormal{coev}}
\DeclareMathOperator{\eval}{\textnormal{ev}}
\DeclareMathOperator{\switch}{\textnormal{switch}}
\newcommand*{\dual}[1]{{#1}^{\vee}}

\let\setminus=\smallsetminus

\input{tikzarrows.tex}

\bibliography{all}

\usepackage{subfiles}

\RedeclareSectionCommand[counterwithin=section,font={\normalsize\normalfont\scshape}]{paragraph}

\begin{document}
\title{A Motivic Snaith Decomposition} \date{}
\author{Viktor Kleen}
\maketitle

\begin{abstract}
  \noindent{\normalsize\normalfont\scshape Abstract}\hskip1em
  The filtration \(\BGL_{0}\subset\dots\subset\BGL_{n-1}\subset\BGL_{n}\) is split by motivic Becker--Gottlieb transfers in the motivic stable homotopy category over any scheme.
  This recovers results by Snaith on the splitting of \(\BGL_{n}(\CC)\) in classical stable homotopy theory by passing to complex realizations.
  On the way, we extend motivic homotopy theory to smooth ind-schemes as bases and show how to construct the necessary fragment of the six operations and duality for this extension.
\end{abstract}

{\footnotesize
  \tableofcontents
}

\section{Introduction}

Snaith \parencite{MR539791,MR517089} described splittings of the natural filtrations on \(\BGL_{n}(\CC)\) and \(\BGL_{n}(\HH)\) in the classical stable homotopy category, using Becker--Gottlieb transfer associated with the normalizer of a maximal torus in \(\GL_{n}\).
He also obtained coarser results for \(\BGL_{n}(\RR)\) and for finite fields.
Using transfers associated with the block-diagonal subgroup \(\GL_{i}\times\GL_{n-i}\subset\GL_{n}\), \textcite{MR1000386} recovered Snaith's results on \(\BGL_{n}(\CC)\) and \(\BGL_{n}(\HH)\) and improve the ones on \(\BGL_{n}(\RR)\) and \(\BGL_{n}(\FF_{q})\).
In fact, the applicability of this approach to the case of \(\BGL_{n}(\CC)\) was first noted by Richter.

In this paper, we prove an analogue of Snaith's, Mitchell and Priddy's and Richter's result in the motivic stable homotopy category of an arbitrary base scheme \(S\):
\begin{theorem}\label{thm:main}
  Over any scheme \(S\), there is a \(\PP^1_{S}\)--stable splitting \(\BGL_{m,+} \wequiv \bigvee_{i=1}^{m} \BGL_i/\BGL_{i-1}\) of the natural filtration \(\BGL_{0}\subset\dots\subset\BGL_{m-1}\subset\BGL_{m}\).
\end{theorem}

Using topological realization, for \(S = \Spec(\CC)\) or \(S = \Spec(\RR)\), this recovers the classical story.
Our approach mirrors the technique of \textcite{MR1000386}, utilizing Becker--Gottlieb transfers associated with the block-diagonal inclusion \(\BGL_{i}\times\BGL_{n-i}\to\BGL_{n}\).
We make essential use of the six operations approach to the motivic Becker--Gottlieb transfer introduced by \textcite{MR3302973}.
The main technical difficulty in our approach to the splitting of \(\BGL_{n,+}\) is the fact that \(\BGL_{n}\) is not representable by a scheme.
However, it is representable by a smooth ind-scheme.

\paragraph*{Overview}
To obtain a transfer for the inclusion \(\BGL_{i}\times\BGL_{n-i}\to\BGL_{n}\) we use an explicit model of \(\BGL_{n}\) as the infinite Grassmannian \(\Gr{n}\).
We give an explicit presentation of the inclusion \(\BGL_{i}\times\BGL_{n-i}\to\BGL_{n}\) as a smooth Zariski--locally trivial bundle \(U\to\Gr{n}\) over the ind-scheme \(\Gr{n}\).
Then we construct a symmetric monoidal presentable stable \infcat \(\SH(\Gr{n})\) which accepts a functor from smooth ind-schemes over \(\Gr{n}\) and supports a functor \(\SH(\Gr{n})\to\SH(S)\).
The motivic Becker--Gottlieb transfer of \(U\to\Gr{n}\) is then defined to be the image in \(\SH(S)\) of the monoidal trace of \(U\) in \(\SH(\Gr{n})\).
We appeal to a result of \textcite{MR1867203} to see that this construction of Becker--Gottlieb transfers has a Mayer--Vietoris property.
This enables us to show that, just like in classical topology, the transfers of the inclusions \(\BGL_{i}\times\BGL_{n-i}\subset\BGL_{n}\) can be used to split the natural filtration on \(\BGL_{n}\).
Since the transfer is a stable phenomenon, this splitting is naturally only found in the stable motivic world.

\paragraph*{Notation}
In what follows \(S\) will be an arbitrary base scheme. For a scheme \(X\) we
write \(\H(X)\) for the \infcat of presheaves of spaces on \(\Sm{X}\) localized
at Nisnevich-local equivalences and projections \(Y\times\AA^1\to Y\). It is a
presentable \infcat in the sense of \parencite{mr2522659}. We refer to \(\H(X)\)
interchangably as the \(\AA^1\)--homotopy category of \(X\) or the motivic
homotopy category of \(X\). The associated pointed \infcat will be denoted by
\(\H_\bullet(X)\). Inverting \((\PP^1,\infty)\in\H_\bullet(X)\) with respect to
the smash product yields the stable motivic homotopy category \(\SH(X)\) of
\(X\). It is a symmetric monoidal, presentable, stable \infcat in the sense of
\parencite{higheralgebra}. An account of this definition of \(\H(X)\),
\(\H_\bullet(X)\) and \(\SH(X)\) for noetherian schemes and its equivalence to
the approach of \parencite{mv} is given in \parencite{MR3281141}, the
generalization to arbitrary schemes can be found in
\parencite[Appendix~C]{MR3302973}.

We follow \parencite{arxiv180610108L} in writing \(X/S\in\SH(S)\) for the
\(\PP^1\)--suspension spectrum of a smooth scheme \(X\) over \(S\). We will
write \(X_+\in \H_\bullet(S)\) for \(X\) with a disjoint basepoint added. We
sometimes do not distinguish notationally between the pointed motivic space
\(X_+\in\H_\bullet(S)\) and its \(\PP^1\)--suspension spectrum \(X/S =
X_+\in\SH(S)\).

When dealing with ind-schemes we have elected not to speak of
\enquote{ind-smooth} schemes and morphisms. Instead, for us a smooth morphisms
between ind-schemes will be what is usually called an ind-smooth morphism,
namely a formal colimit of smooth morphisms.

\paragraph*{Acknowledgements}
I would like to thank Marc Hoyois for countless fruitful conversations and Aravind Asok for his steady advice. Without them this paper would not exist.

\section{Becker--Gottlieb Transfers in Motivic Homotopy Theory}

Becker and Gottlieb introduced their eponymous transfer maps in
\parencite{MR0377873} as a tool for giving a simple proof of the Adams
conjecture. They considered a compact Lie group \(G\) and a fiber bundle \(E \to
B\) over a finite CW complex with structure group \(G\) and whose fiber \(F\) is
a closed smooth manifold with a smooth action by \(G\). There is a smooth
\(G\)-equivariant embedding \(F\subset V\) of \(F\) into a finite dimensional
representation \(V\) of \(G\). There is an associated Pontryagin--Thom collapse
map \(S^V\to F^{\nu}\) where \(\nu\) is the normal bundle of \(F\) in \(V\) and
\(F^\nu\) is its Thom space. Denoting by \(\tau\) the tangent bundle of \(F\)
one obtains a morphism
\[
  S^V \to F^\nu \to F^{\tau\oplus\nu}\wequiv F_+\smashp S^V
\]
in \(G\)-equivariant homotopy theory. Assuming that \(E\to B\) is associated to
a principal \(G\)-bundle \(\widetilde E\to B\) one gets a map
\[
  \widetilde E\times S^V \to \widetilde E\times ( F_+\smashp S^V )
\]
and passing to homotopy orbits with respect to the diagonal \(G\)-actions yields
the transfer map \(B_+ \to E_+\) in the stable homotopy category.

This construction of the transfer was generalized in \parencite{mr656721}. The
map \(S^V \to F_+\smashp S^V\) arises from a \emph{duality datum} in
parameterized stable homotopy theory over the base space \(B\).

\begin{definition}
  A \emph{duality datum} in a symmetric monoidal category consists of a pair of
  objects \(X\) and \(\dual{X}\) with morphisms \(\one{} \to[\coev] X\otimes
  \dual{X}\) and \(\dual{X}\otimes X\to[\ev] \one{}\) such that the compositions
  \begin{align*}
    & X \to[\coev\otimes\id] X\otimes \dual{X}\otimes X \to[\id\otimes\ev] X \\
    \shortintertext{and}
    & \dual{X} \to[\id\otimes\coev] \dual{X} \otimes X \otimes \dual{X} \to[\ev\otimes\id] \dual{X}
  \end{align*}
  are identities. In this situation \(\dual{X}\) is said to be a \emph{right dual}
  of \(X\) and \(X\) is said to be a \emph{left dual} of \(\dual{X}\). If \(X\)
  is additionally a right dual of \(\dual{X}\), then \(X\) is said to be
  \emph{strongly dualizable} with dual \(\dual{X}\).

  A duality datum in a symmetric monoidal \infcat \(\cal C\) is a duality datum
  in the homotopy category \(\ho{\cal C}\), see
  \parencite[section~4.6.1]{higheralgebra}.
\end{definition}

\begin{remark}
  In \parencite{arxiv180610108L}, Levine defines a dual \(\dual{X} =
  \map(X,\one{})\) for any object \(X\) in a \emph{closed} symmetric monoidal
  category. Then \(X\) is called strongly dualizable whenever the induced morphism
  \(\dual{X}\otimes X\to\map(X,X)\) is an equivalence. By
  \parencite[Lemma~4.6.1.6]{higheralgebra} this coincides with our definition.
\end{remark}

Dold and Puppe show that, for a fiber bundle \(E\to B\) with fiber a compact
smooth manifold, there is a duality datum in the homotopy category of
\(B\)--parameterized spectra. It exhibits the fiberwise Thom spectrum of the
fiberwise stable normal bundle to \(E\) as a dual of the suspension spectrum of
\(E\). They then show that the transfer in \parencite{MR0377873} is an instance
of the following general construction.

\begin{definition}
  In a symmetric monoidal \infcat \(\cal C\), suppose that an object \(X\) is
  equipped with a map \(\Delta\colon X\to X\otimes C\) for some other object \(C\).
  Furthermore, suppose that \(X\) is strongly dualizable. The \emph{transfer of
    \(X\) with respect to \(\Delta\)} is defined as the composition
  \[
    \transfer[X,\Delta]\colon \one{} \to[\coev] X\otimes\dual{X} \to[\switch]
    \dual{X}\otimes X\to[\id\otimes\Delta] \dual{X}\otimes X\otimes
    C\to[\eval\otimes\id] \one{}\otimes C\wequiv C.
  \]
  If there can be no risk of confusion we write \(\transfer[X] = \transfer[X,\Delta]\).
\end{definition}

In \autoref{sec:ind-schemes} we construct a symmetric monoidal \infcat
\(\SH(B)\) for every smooth ind-scheme \(B\) over a base scheme \(S\). This
enables us to extend the definition of the motivic Becker--Gottlieb transfer in
\parencite{arxiv180610108L}.

\begin{definition}
  For a smooth map \(f\colon E\to B\) between smooth ind-schemes over \(S\) with
  \(E/B\in\SH(B)\) strongly dualizable we define the \emph{relative transfer}
  \(\Tr{f/B}\colon\one{B}\to E/B\) as follows: Applying \(f_\#\) to the diagonal
  \(E\to E\times_B E\) gives a morphism \(\Delta\colon E/B\to E/B\smashp E/B\)
  in \(\SH(B)\) and we set \(\Tr{E/B} = \Tr{f/B} = \transfer[E/B,\Delta]\).

  Additionally, since \(\pi\colon B\to S\) is a smooth ind-scheme, we can define
  the \emph{absolute transfer} of \(f\) as
  \[
    \Tr{f/S} = \pi_\#(\Tr{f/B})\colon E/S\to B/S.
  \]
\end{definition}

\begin{proposition}\label{prop:transfer-props}
  The motivic Becker--Gottlieb transfer enjoys the following properties.
  \begin{enumerate}
  \item The transfer is additive in homotopy pushouts: Suppose \(X\), \(Y\),
    \(U\) and \(V\) are smooth ind-schemes over a smooth ind-scheme \(B\) over
    \(S\). Further suppose that there is a homotopy cocartesian square
    \[
      \begin{tikzcd}
        X/B \ar[r] \ar[d] & U/B \ar[d] \\
        V/B \ar[r] & Y/B
      \end{tikzcd}
    \]
    in \(\SH(B)\). Assume that \(Y/B\), \(U/B\) and \(V/B\) are strongly
    dualizable. Then \(\Tr{Y/B}\) is a sum of the compositions
    \begin{align*}
      & \one{B} \to[{\Tr{U/B}}] U/B \to Y/B \\
      & \one{B} \to[{\Tr{V/B}}] V/B \to Y/B \\
      \shortintertext{and}
      & \one{B} \to [{\Tr{X/B}}] X/B \to Y/B
    \end{align*}
    in \(\SH(B)\).
  \item The relative transfer is compatible with pullback: If \(p\colon B' \to
    B\) and \(f\colon E\to B\) are maps of smooth ind-schemes over \(S\) and
    \(E/B\) is strongly dualizable in \(\SH(B)\) then the pullback
    \(p^*(E/B)\wequiv (E\times_B B')/B'\) is strongly dualizable in \(\SH(B')\)
    and \(\Tr{p^*f\,/B'} \wequiv p^*\Tr{f/B}\).
  \item The absolute transfer is natural in cartesian squares: If
    \[
      \begin{tikzcd}
        E' \ar[r]\ar[d, "f'"'] & E \ar[d, "f"] \\
        B' \ar[r] & B
      \end{tikzcd}
    \]
    is a cartesian square of smooth ind-schemes over \(S\) and the vertical maps are
    smooth, then the square
    \[
      \begin{tikzcd}
        E'/S \ar[r] & E/S \\
        B'/S \ar[r] \ar[u, "{\Tr{f'/S}}"] & B/S \ar[u, "{\Tr{f/S}}"']
      \end{tikzcd}
    \]
    commutes in \(\SH(S)\).
  \end{enumerate}
\end{proposition}

To prove part (i) of \autoref{prop:transfer-props} we appeal to a general
additivity result of May's. In the context of symmetric monoidal triangulated
categories, \parencite{MR1867203} proves that the transfer is additive in
distinguished triangles. However, since duality in symmetric monoidal \infcats
is characterised at the level of homotopy categories, May's theorem admits the
following reformulation.

\begin{theorem}[{\parencite[Theorem~1.9]{MR1867203}}]\label{thm:transfer-additivity}
  Let \(\cal C\) be a symmetric monoidal stable \infcat and let \(X\to Y\to Z\)
  be a cofiber sequence in \(\cal C\). Assume \(C\in\cal C\) is such that
  \(\_{\otimes}C\) preserves cofiber sequences. Suppose that \(Y\) is equipped
  with a map \(\Delta_Y\colon Y\to Y\otimes C\) and that \(X\) and \(Y\) are
  strongly dualizable. Then \(Z\) is strongly dualizable and there are maps
  \(\Delta_X\) and \(\Delta_Z\) such that
  \[
    \begin{tikzcd}
      X \ar[r] \ar[d, "\Delta_X"] & Y \ar[r] \ar[d,"\Delta_Y"] & Z \ar[d,"\Delta_Z"] \\
      X\otimes C \ar[r] & Y\otimes C \ar[r] & Z\otimes C
    \end{tikzcd}
  \]
  commutes. Furthermore, we have \(\transfer[Y,\Delta_Y] = \transfer[X,\Delta_X] +
  \transfer[Z,\Delta_Z]\) in \(\pi_0\map_{\cal C}(\one, C)\).
\end{theorem}

\begin{proof}[Proof of \autoref{prop:transfer-props}, (i)]
  The homotopy cocartesian square induces a cofiber sequence
  \[
    U/B \vee V/B \to Y/B \to S^1 \smashp X/B
  \]
  in \(\SH(B)\). Shifting this sequence yields and introducing diagonal maps
  gives a diagram
  \[
    \begin{tikzcd}
      X/B \ar[r] \ar[d, "\Delta_X"] & U/B \vee V/B \ar[r] \ar[d,
      "\Delta_U\vee\Delta_V"] & Y/B \ar[dd, "\Delta_Y"] \\
      X/B\smashp X/B \ar[d, "(1)"] & (U/B\smashp U/B) \vee (V/B\smashp V/B)
      \ar[d, "(2)"] & {}\\
      X/B\smashp Y/B \ar[r] & (U/B\vee V/B)\smashp Y/B \ar[r] & Y/B\smashp Y/B
    \end{tikzcd}
  \]
  in which the outer two rows are cofiber sequences and the maps \((1)\) and
  \((2)\) are induced from the maps \(X/B\to Y/B\), \(U/B\to Y/B\) and \(V/B\to
  Y/B\) respectively. Then we can conclude using \autoref{thm:transfer-additivity}.
\end{proof}

Part (ii) of \autoref{prop:transfer-props} is proven in
\parencite[Lemma~1.6]{arxiv180610108L}. We formulate the proof of part (iii) as
a lemma.

\begin{lemma}\label{lem:transfer-natural}
  Let \(S\) be a scheme and \(B\) and \(B'\) smooth ind-schemes over \(S\).
  Suppose that \(f\colon E\to B\) is smooth with \(E/B\in\SH(B)\) strongly
  dualizable and
  \[
    \begin{tikzcd}
      E' \ar[r, "i'"]\ar[d, "f'"] & E \ar[d, "f"] \\
      B' \ar[r, "i"'] & B
    \end{tikzcd}
  \]
  is cartesian. Then the square
  \[
    \begin{tikzcd}
      E'/S \ar[r, "i'/S"] & E/S \\
      B'/S \ar[r, "i/S"'] \ar[u, "{\Tr{f'/S}}"] & B/S \ar[u, "{\Tr{f/S}}"']
    \end{tikzcd}
  \]
  is homotopy commutative.
\end{lemma}
\begin{proof}
  Write \(p\colon B'\to S\) and \(q\colon B\to S\) for the structure morphisms.
  There is a natural transformation \(p_\#i^*\to q_\#\) defined as the
  composition
  \[
    p_\#i^* \to[\unit] p_\#i^*q^*q_\# \wequiv p_\# p^*q_\# \to[\counit] q_\#.
  \]
  Consequently, we obtain a homotopy commutative diagram
  \[
    \begin{tikzcd}[column sep=5em]
      \mathllap{B'/S = {}}p_\#p^*\one{S} \ar[r, "p_\#\Tr{f'/B'}"] \ar[d, equal] & p_\#f'_\#
      {i'}^*f^*q^*\one{S}\mathrlap{{} = E'/S}\ar[d, "\Ex_\#^*"] \\
      p_\#i^*q^*\one{S} \ar[r, "p_\#i^*\Tr{f/B}"] \ar[d] &
      p_\#i^*f_\#f^*q^*\one{S} \ar[d] \\
      \mathllap{B/S = {}}q_\# q^*\one{S} \ar[r, "q_\#\Tr{f/B}"'] & q_\#
      f_\#f^*q^*\one{S}\mathrlap{{} = E/S.}
    \end{tikzcd}
  \]
  Chasing through the definition of \(p_\#i^*\to q_\#\) shows that the leftmost
  composite vertical map is \(i/S\) and that the rightmost vertical map is \(i'/S\).
\end{proof}

Finally, we will need some tools to understand when a smooth morphism \(E\to B\)
of smooth ind-schemes over \(S\) determines a strongly dualizable object \(E/B\)
in \(\SH(B)\). We have the following formulation of motivic Atiyah duality.
\begin{theorem}[{see \parencite{voevodsky-cross-functors}, \parencite{MR2135324}, \parencite{ayoubi}, \parencite{2009arXiv0912.2110C}}]\label{thm:atiyah-duality}
  If \(Y\to X\) is a smooth and proper morphism of schemes, then
  \(Y/X\in\SH(X)\) is strongly dualizable.
\end{theorem}

Because the property of being strongly dualizable is formulated in the homotopy
category, it is immediate that any smooth scheme \(Y\to X\) such that \(Y/X\) is
\(\AA^1\)-homotopy equivalent to a smooth and proper scheme over \(X\) defines
a strongly dualizable object in \(\SH(X)\). Furthermore, dualizability is local
in the following sense.

\begin{theorem}[{\parencite[Proposition~1.2, Theorem~1.10]{arxiv180610108L}}]\label{thm:local-dualizability}
  Let \(B\) be a scheme over \(S\). Suppose that \(E\in\SH(B)\) and there is a
  finite Nisnevich covering family \(\{j_i\colon U_i\to B\}\) and that \(j_i^*E
  \in\SH(U_i)\) is strongly dualizable. Then \(E\) is strongly dualizable as well.

  If \(E\to B\) is a Nisnevich--locally trivial fiber bundle with smooth fiber
  \(F\) and \(B\) is smooth over \(S\), then \(E/B\) is strongly dualizable in
  \(\SH(B)\) if \(F/S\) is strongly dualizable in \(\SH(S)\)
\end{theorem}

\section{Transfers of Grassmannians}

\begin{definition}
  The ind-scheme \(\Gr{r}\) is the sequential colimit of the Grassmannians
  \(\Gr[n]{r}\) of \(r\)--planes in \(n\)--space along the canonical closed immersions
  \(\Gr[n]{r}\clim\Gr[n+1]{r}\).
\end{definition}

It is well known that \(\Gr{r}\) is a model for \(\BGL_r\) in the
\(\AA^1\)--homotopy category. In fact, let \(U_r(N)\) be the scheme of
monomorphisms \(\cal O^r \to \cal O^N\). Along \(\cal O^N\oplus 0 \subset \cal
O^{N+1}\), there are closed embeddings \(U_r(N)\clim U_r(N+1)\) and
\parencite[Proposition~4.3.7]{mv} shows that the colimit \(U_r(\infty) =
\colim_N U_r(N)\) along these embeddings is contractible in \(\H(S)\). Also, the
quotient \(U_r(N)/\GL_r\) is isomorphic to \(\Gr[N]{r}\) and consequently
\(U_r(\infty)/\GL_r\isom \Gr{r}\) is a model for \(\BGL_r\).

Direct sum defines a morphism \(U_r(N)\times U_{n-r}(N)\to U_n(2N)\) which is
equivariant with respect to the block diagonal inclusion
\(\GL_r\times\GL_{n-r}\to \GL_n\). Passing to the colimit \(N\cto\infty\) and
taking quotients yields a morphism
\[
  i_{r,n}\colon\Gr{r}\times\Gr{n-r} \to \Gr{n}.
\]
This morphism is equivalent in \(\H(S)\) to the map
\(\BGL_r\times\BGL_{n-r}\to\BGL_n\) induced by the block diagonal inclusion
\(\GL_r\times\GL_{n-r}\subset \GL_n\). The goal of this section is to develop a
partial inductive description of the absolute transfer \(\transfer[n,r]\colon
\Gr{n,+}\to\Gr{r,+}\smashp\Gr{n-r,+}\) of \(i_{r,n}\) in \(\SH(S)\). For this
purpose a different version of \(i_{r,n}\) in \(\H(S)\) will be more convenient.

\begin{lemma}
  In \(\H(S)\) there is an equivalence \(\Gr{r}\times\Gr{n-r}\to
  U_n(\infty)/(\GL_r\times\GL_{n-r})\). Along this equivalence, \(i_{r,n}\)
  corresponds to the quotient
  \[
    \overline{i_{r,n}}\colon U_n(\infty)/(\GL_r\times\GL_{n-r}) \to U_n(\infty)/\GL_n\isom \Gr{n}
  \]
  by \(\GL_n\).
\end{lemma}
\begin{proof}
  Writing \(\varphi\colon U_r(N)\times U_{n-r}(N)\to U_n(2N)\) for the map
  induced by taking direct sums, we obtain a commutative diagram
  \[
    \begin{tikzcd}
      U_r(N)\times U_{n-r}(N) \ar[r, "\varphi"]\ar[d] & U_n(2N)\ar[d] \\
      \Gr[N]{r}\times\Gr[N]{n-r} \ar[r]\ar[dr, "i_{r,n}"'] & U_n(2N)/(\GL_r\times\GL_{n-r})\ar[d] \\
      & \Gr[2N]{n}.
    \end{tikzcd}
  \]
  Passing to the colimit \(N\cto\infty\) the horizontal maps become equivalences.
\end{proof}

\begin{lemma}\label{lem:gr-fiber-bundle}
  The morphism \(\overline{i_{r,n}}\) is a Zariski--locally trivial bundle over
  \(\Gr{n}\). Its fiber is the quotient \(\GL_n/(\GL_r\times\GL_{n-r})\).
\end{lemma}
\begin{proof}
  By construction, the morphism \(\overline{i_{r,n}}\) is isomorphic to the
  colimit of the quotient maps \(U_n(N)/(\GL_r\times\GL_{n-r}) \to
  U_n(N)/\GL_n\isom \Gr[N]{n}.\) But these are all Zariski--locally trivial with
  fiber \(\GL_n/(\GL_r\times\GL_{n-r})\).
\end{proof}

We note that \(\GL_n/(\GL_r\times\GL_{n-r})\) is equivalent to \(\Gr[n]{r}\) in
\(\H(S)\) and this equivalence is compatible with the respective \(\GL_n\)
actions. This is shown in \parencite[Lemma~3.1.5]{MR3748687} and
implies in particular that the image in \(\SH(\Gr{n})\) of the associated
bundle \(U_n(\infty)\times^{\GL_n}\Gr[n]{r}\to\Gr{n}\) is equivalent to that of the
quotient \(U_n(\infty)/(\GL_r\times\GL_{n-r})\to\Gr{n}\).

\begin{lemma}
  The morphism \(\overline{i_{r,n}}\colon U_n(\infty)/(\GL_r\times\GL_{n-r})
  \to\Gr{n}\) defines a strongly dualizable object
  \(G_{r,n}\in\SH(\Gr{n})\).
\end{lemma}
\begin{proof}
  By \autoref{lem:ind-dualizability} it will be enough to show that the
  pullback \(i\colon E\to \Gr[N]{n}\) of \(\overline{i_{r,n}}\) along the inclusion
  \(\Gr[N]{n}\to\Gr{n}\) defines a dualizable object in \(\SH(\Gr[N]{n})\) for
  all \(N\). But, by \autoref{lem:gr-fiber-bundle} the morphism \(i\) is a
  Zariski-locally trivial fiber bundle over \(\Gr[N]{n}\) with fiber
  \(X = \GL_n/(\GL_r\times\GL_{n-r})\). Hence, to show that \(i\) defines a strongly
  dualizable object in \(\SH(\Gr[N]{n})\), by \autoref{thm:local-dualizability}
  it is enough to show that \(X/S\in\SH(S)\) is strongly dualizable.

  But we have seen that \(X\wequiv\Gr[n]{r}\) in \(\H(S)\) and therefore also in
  \(\SH(S)\). The scheme \(\Gr[n]{r}\) is smooth and proper over \(S\), so
  motivic Atiyah duality, \autoref{thm:atiyah-duality}, implies that
  \(\Gr[n]{r}/S\) and therefore also \(X/S\) is strongly dualizable in
  \(\SH(S)\), see for example \parencite[Proposition~1.2]{arxiv180610108L}.
\end{proof}

\begin{lemma}\label{lem:grassmann-decomp}
  Suppose \(r<n\). The open complement of the closed immersion
  \(\Gr[n-1]{r}\clim \Gr[n]{r}\) is the total space of an affine space bundle of
  rank \(n-r\) over \(\Gr[n-1]{r-1}\).

  Dually, the complement of the closed immersion \(\Gr[n-1]{r-1}\clim
  \Gr[n]{r}\) is the total space of an affine space bundle of rank \(r\) over
  \(\Gr[n-1]{r}\).
\end{lemma}
\begin{proof}
  Suppose \(\Spec(A)\) is an affine scheme mapping to \(S\). On
  \(\Spec(A)\)--valued points, the inclusion \(\Gr[n-1]{r}\clim\Gr[n]{r}\) is
  given by considering a projective submodule \(P\) of \(A^{n-1}\) as a
  submodule of \(A^n = A^{n-1}\oplus A\). It follows that the complement \(U\)
  of \(\Gr[n-1]{r}\) has \(\Spec(A)\)--valued points
  \[
    U(\Spec A) = \{P\subset A^n : \text{\(P\) is projective of rank \(r\) and
      \(P \not\subset A^{n-1}\oplus 0\)}\}.
  \]
  Given \(P\in U(\Spec A)\), the module \(P\cap (A^{n-1}\oplus 0)\) will be
  locally free of rank \(r-1\). This gives a map \(\varphi\colon U \to
  \Gr[n-1]{r-1}\) which is trivial over the standard Zariski--open cover of
  \(\Gr[n-1]{r-1}\) with fiber \(\AA^{n-r}\).

  The dual statement is proved similarly. In fact, the bundle \(V \to
  \Gr[n-1]{r}\) in question is the tautological \(r\)-plane bundle on
  \(\Gr[n-1]{r}\).
\end{proof}

The decomposition \(\Gr[n]{r} = U\cup V\) of the last lemma yields a
homotopy cocartesian square
\[
  \begin{tikzcd}
    \mathllap{U\setminus \Gr[n-1]{r-1} ={}} U\cap V \ar[r] \ar[d] &
    V\mathrlap{{}\wequiv \Gr[n-1]{r}} \ar[d] \\
    \mathllap{\Gr[n-1]{r-1}\wequiv {}}U \ar[r] & \Gr[n]{r}
  \end{tikzcd}
\]
in the \(\AA^1\)--homotopy category \(\H(S)\). It is immediate that this
decomposition of \(\Gr[n]{r}\) is stable under the action of \(\GL_{n-1}\times
1\subset \GL_n\). We can therefore pass to the bundles over \(\Gr{n-1}\)
associated to the universal \(\GL_{n-1}\)--torsor \(U_{n-1}(\infty)\) over
\(\Gr{n-1}\) and obtain a homotopy cocartesian square
\[
  \begin{tikzcd}
    (U_{n-1}(\infty)\times^{\GL_{n-1}} (U\cap V))/\Gr{n-1}\ar[r]\ar[d] & G_{r,n-1}\ar[d] \\
    G_{r-1,n-1}\ar[r] & (U_{n-1}(\infty)\times^{\GL_{n-1}}\Gr[n]{r})/\Gr{n-1}
  \end{tikzcd}
\]
in \(\SH(\Gr{n-1})\).

\begin{proposition}\label{prop:transfer-decomp}
  Suppose \(r<n\) and consider the composition
  \[
    \varphi\colon \Gr{n-1,+} \to[\incl] \Gr{n+} \to[{\transfer[n,r]}]\Gr{r+}\smashp\Gr{n-r,+}
  \]
  where \(\incl\) is given by the assignment \(P\mapsto P\oplus A\) on \(\Spec(A)\)--valued
  points. Then there is a map \(\psi\colon \Gr{n-1,+}\to
  \Gr{r-1,+}\smashp\Gr{n-r,+}\) in \(\SH(S)\) such that \(\varphi\) is the sum
  of the compositions
  \begin{align*}
    \Gr{n-1,+} \to[{\transfer[n-1,r]}]\Gr{r,+}\smashp\Gr{n-1-r,+} \to[\id\smashp\incl]\Gr{r,+}\smashp\Gr{n-r,+} \\
    \Gr{n-1,+} \to[{\transfer[n-1,r-1]}]\Gr{r-1,+}\smashp\Gr{n-r,+}\to[\incl\smashp\id]\Gr{r,+}\smashp\Gr{n-r,+} \\
\shortintertext{and}
    \Gr{r-1,+}\to[\psi]\Gr{r-1,+}\smashp\Gr{n-r,+} \to[\incl\smashp\id]\Gr{r,+}\smashp\Gr{n-r,+}.
  \end{align*}
\end{proposition}
\begin{proof}
  Consider the homotopy pullback
  \[
    \begin{tikzcd}
      \mathllap{E={}}U_{n-1}(\infty)\times^{\GL_{n-1}}\Gr[n]{r}\ar[d]\ar[r] & \Gr{r}\times\Gr{n-r}\ar[d] \\
      \Gr{n-1}\ar[r, "\incl"'] & \Gr{n}
    \end{tikzcd}
  \]
  in \(\H(S)\). By the discussion following \autoref{lem:grassmann-decomp} we
  obtain a cofiber sequence
  \[
    X/\Gr{n-1} \to G_{r,n-1}\vee G_{r-1,n-1} \to E/\Gr{n-1}
  \]
  in \(\SH(\Gr{n-1})\) where \(X = U_{n-1}(\infty)\times^{\GL_{n-1}}(U\cap V)\).
  \Autoref{thm:transfer-additivity} then shows that
  \[
    \transfer[E/\Gr{n-1}] = \transfer[G_{r,n-1}] + \transfer[G_{r-1,n-1}] - \transfer[X/\Gr{n-1}]
  \]
  in \(\SH(\Gr{n-1})\). Passing to the absolute transfer and using
  \autoref{lem:transfer-natural} yields that \(\varphi\) is the sum of
  the compositions
  \begin{align*}
    &\Gr{n-1,+} \to[{\transfer[n-1,r]}] \Gr{r,+}\smashp\Gr{n-1-r,+} \to[\id\smashp\incl]\Gr{r,+}\smashp\Gr{n-r,+} \\
    &\Gr{n-1,+} \to[{\transfer[n-1,r-1]}] \Gr{r-1,+}\smashp\Gr{n-r,+} \to[\incl\smashp\id]\Gr{r,+}\smashp\Gr{n-r,+} \\
\shortintertext{and}
    &\Gr{n-1,+}\to X_+ \to \Gr{r,+}\smashp\Gr{n-r,+}
  \end{align*}
  in \(\SH(S)\). Here, the map \(X_+\to\Gr{r,+}\smashp\Gr{n-r,+}\) is obtained
  from the inclusion \(U\cap V\subset \Gr[n]{r}\) by passing to associated
  bundles. Now, this inclusion factors through the inclusion of \(U\) into
  \(\Gr[n]{r}\). By \autoref{lem:grassmann-decomp} the inclusion
  \(\Gr[n-1]{r-1}\subset U\) is an \(\AA^1\)--equivalence, being the zero
  section of an affine space bundle. Therefore
  \(X_+\to\Gr{r,+}\smashp\Gr{n-r,+}\) factors through the map
  \(\incl\smashp\id\colon\Gr{r-1,+}\smashp\Gr{n-r,+}\to\Gr{r,+}\smashp\Gr{n-r,+}\).
  This way we obtain the map \(\psi\) and the required decomposition of \(\transfer[n,r]\circ\incl\).
\end{proof}

\section{Proof of the Theorem}

We have the filtration
\[
    \Gr{0,+}\to[i_1] \Gr{1,+} \to[i_2] \dots \to[i_n] \Gr{n,+} \to
    \dots\to[i_m] \Gr{m,+}
\]
and we have seen that for \(r \leq n\) the map
\(i_{r,n}\colon\Gr{r}\times\Gr{n-r}\to \Gr{n}\) admits an absolute transfer
\(\transfer[n,r]\colon\Gr{n,+}\to \Gr{r,+}\smashp \Gr{n-r,+}\) in the motivic
stable homotopy category \(\SH(S)\). Write \(f_{n,r}\colon \Gr{n,+}\to
\Gr{r,+}\) for the composition
\[
  \Gr{n,+} \to[{\transfer[n,r]}] \Gr{r,+}\smashp\Gr{n-r,+} \to[\proj] \Gr{r,+}
\]
and \(\phi_{n,r}\) for the composition
\[
  \Gr{n,+}\to[f_{n,r}] \Gr{r,+}\to\Gr{r}/{\Gr{r-1}}.
\]

\begin{lemma}\label{lem:phi-comp}
  With notation as above, for \(r < n\) the compositions
\[
  \Gr{n-1,+} \to[i_n] \Gr{n,+} \to[f_{n,r}] \Gr{r,+} \to \Gr{r}/{\Gr{r-1}}
\]
and
\[
  \Gr{n-1,+} \to[f_{n-1,r}] \Gr{r,+} \to \Gr{r}/{\Gr{r-1}}
\]
coincide.
\end{lemma}
\begin{proof}
  By \autoref{prop:transfer-decomp} the composition \(f_{n,r}\circ i_n\) is a
  sum of two compositions
  \begin{align*}
    &\Gr{n-1,+} \to[{\transfer[n-1,r]}] \Gr{r,+}\smashp\Gr{n-1-r,+} \to[\id\smashp\incl]\Gr{r,+}\smashp\Gr{n-r,+}\to[\proj]\Gr{r,+} \\
    \shortintertext{and}
    &\Gr{n-1,+}\to\Gr{r-1,+}\smashp\Gr{n-r,+}\to[\incl\smashp\id]\Gr{r,+}\smashp\Gr{n-r,+}\to[\proj]\Gr{r,+}
  \end{align*}
  in \(\SH(S)\). But the composition
  \[
    \Gr{r-1,+}\to[\incl]\Gr{r,+} \to \Gr{r}/{\Gr{r-1}}
  \]
  vanishes. Therefore, \(f_{n,r}\circ i_n\) coincides with the composition
  \[
    \Gr{n-1,+}\to[f_{n-1,r}]\Gr{r,+}\to\Gr{r}/{\Gr{r-1}}
  \]
  in \(\SH(S)\).
\end{proof}

\begin{proof}[{Proof of \autoref{thm:main}}]
  Proceeding by induction on \(n\), assume that
\[
  \Phi = \bigvee_{r=0}^{n-1} \phi_{n-1,r}\colon \Gr{n-1,+} \to \bigvee_{r=0}^{n-1} \Gr{r}/{\Gr{r-1}}
\]
is an equivalence in \(\SH(S)\). Because of \autoref{lem:phi-comp} we have a
commutative diagram
\[
  \begin{tikzcd}
    {} & \Gr{n,+}\ar[dr, "\Phi'"] & {} \\
    \Gr{n-1,+} \ar[rr, "\Phi"'] \ar[ur, "i_n"] & {} & \displaystyle\bigvee_{r=0}^{n-1}\Gr{r}/{\Gr{r-1}}
  \end{tikzcd}
\]
where \(\Phi' = \bigvee_{r=0}^{n-1}\phi_{n,r}\). It follows that \(\Phi^{-1}\circ\Phi'\circ i_n \wequiv \id\),
\ie~\(i_n\) admits a left inverse. That is to say, the cofiber sequence
\[
  \Gr{n-1,+} \to[i_n]\Gr{n,+} \to \Gr{n}/\Gr{n-1}
\]
splits and yields an equivalence
\[
  \Gr{n,+} \to[{(\Phi^{-1}\Phi') \vee \phi_{n,n}}] \Gr{n-1,+} \vee \Gr{n}/{\Gr{n-1}}
\]
since \(\phi_{n,n}\) is by definition the canonical projection. Post-composing
with \(\Phi\vee{\id}\) then shows that the stable map
\(\Phi'\vee\phi_{n,n}\colon \Gr{n,+}\to\bigvee_{r=0}^n \Gr{r}/{\Gr{r-1}}\)
is an equivalence in \(\SH(S)\) as well.
\end{proof}

\appendix
\section{Stable Motivic Homotopy Theory of Smooth Ind-Schemes}\label{sec:ind-schemes}

We freely use the theory of presentable \infcats as developed in
\parencite[section~5.5.3]{mr2522659}. The \infcat of presentable \infcats with
left adjoints as morphisms is denoted \(\PrL\) while the \infcat of presentable
\infcats with right adjoints as morphisms is denoted \(\PrR\). There is an
equivalence \(\PrL\wequiv(\PrR)^\Op\) of \infcats which is the identity on
objects and sends a left adjoint functor to its right adjoint. Both \(\PrL\) and
\(\PrR\) are complete and cocomplete and the homotopy limits in both \(\PrL\) and
\(\PrR\) coincide with homotopy limits in the \infcat of \infcats.

\begin{definition}
  A \emph{smooth ind-scheme} over \(S\) is an object of \(\Ind(\Sm{S})\), the \infcat of
  ind-objects in the category of smooth schemes over \(S\) with arbitrary
  morphisms between them. A morphism of ind-schemes is smooth if it can be presented
  as a colimit of smooth morphisms in \(\Sm{S}\).
\end{definition}

The goal of this section will be to generalize the definition of the stable
motivic homotopy category \(\SH\) to smooth ind-schemes over \(S\). Our approach
is to use part of the six functor formalism for \(\SH\), as established in
\parencite{ayoubi,ayoubii} for noetherian schemes and extended to arbitrary
schemes in \parencite[Appendix~C]{MR3302973}. An overview of the standard
functorialities, at least at the level of triangulated categories, can be found
in \parencite{2009arXiv0912.2110C}.

The first functoriality of \(\SH\) can be summarized as follows. For every
morphism \(f\colon X\to Y\) between smooth schemes over \(S\) we have an
adjunction
\[
  \adjunction{f^*}{\SH(X)}{\SH(Y)}{f_*}
\]
between the stable presentable \infcats \(\SH(X)\) and \(\SH(Y)\). These
adjunctions assemble into functors \(\SH^*\colon \Sm{S}^\Op \to \PrL\) and
\(\SH_*\colon \Sm{S} \to \PrR\) which are naturally equivalent after composing
with the equivalence \(\PrL\wequiv (\PrR)^\Op\). If \(f\colon X\to Y\) is smooth,
then there is an additional adjunction
\[
  \adjunction{f_\#}{\SH(Y)}{\SH(X)}{f^*}.
\]
These assemble into a functor \(\SH_\#\colon \Sm{S,\sm}\to\PrL\) from the wide
subcategory of \(\Sm{S}\) consisting of smooth morphisms between smooth schemes
over \(S\). There are various exchange transformations associated with a
cartesian square
\[
  \begin{tikzcd}
    \bullet \ar[r, "g"] \ar[d, "q"'] & \bullet \ar[d, "p"] \\
    \bullet \ar[r, "f"'] & \bullet
  \end{tikzcd}
\]
in \(\Sm{S}\), of which we only mention the transformation
\[
  \Ex_\#^*\colon g_\#q^* \to p^*f_\#
\]
when \(f\) and hence \(g\) is smooth. More details on these exchange
transformations may be found in \parencite{2009arXiv0912.2110C}.

Because \(\PrR\) is cocomplete, the functor \(\SH_*\) naturally extends to a
functor
\[
  \SH_*\colon \Ind(\Sm{S}) \to \PrR
\]
and we obtain a functor
\[
  \SH^*\colon \Ind(\Sm{S})^\Op \to \PrL
\]
by again composing with the equivalence \(\PrL\wequiv(\PrR)^\Op\).

More explicitly, if \((X_i)_{i\in I}\) is a filtered diagram of smooth schemes
over \(S\) and \(X = \colim_i X_i\) as an ind-scheme over \(S\), then
\[
  \SH^*(X) = \holim_i \SH^*(X_i)\quad\text{and}\quad\SH_*(X) = \hocolim_i \SH_*(X_i).
\]
Note that \(\SH^*(X)\) and \(\SH_*(X)\) are equivalent \infcats since homotopy
limits along left adjoints in \(\PrL\) correspond to homotopy colimits along
their right adjoints in \(\PrR\), see \parencite[section~5.5.3]{mr2522659}. This
description of \(\SH(X)\) also shows that it inherits the structure of a closed
symmetric monoidal, stable, presentable \infcat, see \parencite[section~3.4.3,
Proposition~4.8.2.18]{higheralgebra}.

The adjunction \(f^*\dashv f_*\) for a morphism \(f\colon X\to Y\) of
ind-schemes is obtained by presenting \(f\) as a colimit of maps \(f_i\colon
X_i\to Y_i\) between schemes over \(S\) and then taking \(f^*\) to be the
functor induced on the homotopy limits in \(\PrL\) and \(f_*\) the functor
induced on the homotopy colimits in \(\PrR\).

It remains to construct the extra left-adjoint \(f_\#\) for a smooth map \(f\)
between ind-schemes over \(S\). First, a morphism \(f\colon X\to Y\) between
ind-schemes is smooth if and only if it is a filtered colimit of smooth maps
\(f_i\colon X_i \to Y_i\). Each \(f_i^*\) admits a left adjoint \(f_{i\#}\) and
since \(\PrR\) is stable under limits, the functor \(f^*\colon
\SH^*(Y)\to\SH^*(X)\) admits a left adjoint as well. That is to say,
\(\SH^*\colon\Ind(\Sm{S})^\Op\to\PrL\) restricts to a functor
\(\SH^*\colon\Ind(\Sm{S})_{\sm}^\Op \to \PrR\) from the wide subcategory of
\(\Ind(\Sm{S})\) consisting of smooth maps between smooth ind-schemes over
\(S\). Composing with the equivalence \(\PrL\wequiv (\PrR)^\Op\) then yields the
functor
\[
  \SH_\#\colon\Ind(\Sm{S})_\sm\to\PrL.
\]
In summary, we have the following proposition.
\begin{proposition}
  For every ind-scheme \(X\) over \(S\), there is a closed symmetric monoidal,
  stable, presentable \infcat \(\SH(X)\). For every morphism \(f\colon X\to Y\)
  between ind-schemes there is an associated adjunction
  \[
    \adjunction{f^*}{\SH(Y)}{\SH(X)}{f_*}
  \]
  with \(f^*\) a monoidal functor. If \(f\) is smooth then there is an
  additional adjunction
  \[
    \adjunction{f_\#}{\SH(X)}{\SH(Y)}{f^*}.
  \]
  These data are functorial in \(f\) and admit various natural exchange
  transformations. If \(X\) happens to be a smooth scheme over \(S\) then this
  version of \(\SH(X)\) is naturally equivalent to the usual construction.
\end{proposition}

Following \parencite{arxiv180610108L}, for a smooth morphism \(f\colon X\to Y\)
of ind-schemes over \(S\) we define \(X/Y = f_\#(\one{X})\in\SH(Y)\) where
\(\one{X}\) denotes the monoidal unit in \(\SH(X)\). In particular, if \(Y=S\),
we see that any smooth ind-scheme \(X\) over \(S\) determines an object
\(X/S\in\SH(S)\). If \(X\) is a smooth scheme over \(S\), then \(X/S\) is
canonically equivalent to the \(\PP^1\)--suspension spectrum of \(X\) in
\(\SH(S)\); see \parencite[Lemma~C.2]{MR3205601}.

\begin{lemma}\label{lem:ind-dualizability}
  Suppose \(B\) is a smooth ind-scheme over \(S\) and \(E\in\SH(B)\). If \(B\)
  is presented as a filtered colimit \(B = \colim_i B_i\) of smooth schemes in
  \(\Ind(\Sm{S})\), let \(f_i\colon B_i\to B\) be the canonical map for each
  \(i\). Then \(E\in\SH(B)\) is strongly dualizable if and only if
  \(f_i^*E\in\SH(B_i)\) is strongly dualizable for every \(i\).
\end{lemma}
\begin{proof}
  This follows from \parencite[Proposition~4.6.1.11]{higheralgebra} since we
  have \(\SH(B)\wequiv \lim_i\SH(B_i)\).
\end{proof}

\begin{proposition}\label{prop:objects-of-slash-S}
  Suppose an ind-scheme \(X\) is presented as a colimit \(X = \colim_i X_i\) in
  \(\Ind(\Sm{S})\). Then there is a natural equivalence \(X/S\wequiv \hocolim_i
  X_i/S\) in \(\SH(S)\).
\end{proposition}
\begin{proof}
  Write \(\pi\colon X\to S\) and \(\pi_i\colon X_i\to S\) for the structure
  morphisms. Suppose \(Y\in\SH(S)\) is arbitrary. Then we have natural
  equivalences
  \begin{align*}
    \map_{\SH(S)}(\pi_\#\one{X}, Y) &\wequiv \map_{\SH(X)}(\one{X}, \pi^*Y) \\
                                    &\wequiv \holim_i\map_{\SH(X_i)}(\one{X_i}, \pi_i^*Y) \\
                                    &\wequiv \holim_i \map_{\SH(S)}(\pi_{i\#}\one{X_i}, Y) \\
                                    &\wequiv \map_{\SH(S)}(\hocolim_i X_i/S, Y)
  \end{align*}
  of mapping spaces. The Yoneda lemma implies that \(X/S =\pi_\#\one{X}\wequiv \hocolim_i
  X_i/S\) in \(\SH(S)\).
\end{proof}

This proposition allows us to extend the definition of the functor
\(\_/S\colon\Sm{S}\to\SH(S)\) in \parencite{arxiv180610108L} to ind-schemes. The
functor \(\_/S\colon\Sm{S}\to\SH(S)\) extends uniquely up to natural equivalence
to a functor \(\_/S\colon \Ind(\Sm{S})\to\SH(S)\) because \(\SH(S)\) is
cocomplete. By \autoref{prop:objects-of-slash-S} this coincides on objects with
the previous construction \(\pi_\#(\one{X})\) for a smooth ind-scheme
\(\pi\colon X\to S\).

\printbibliography

\end{document}